\documentclass[reqno, 12pt]{amsart}
\usepackage[letterpaper,hmargin=1in,vmargin=1in]{geometry}
\usepackage{amsmath, amssymb, amsthm, verbatim, bbm, url} 
\usepackage{graphicx}
\usepackage{enumerate}

\newcommand \N {\mathbb{N}}
\newcommand \R {\mathbb{R}}

\newcommand \Sc {\mathcal{S}}

\newcommand \Oh {\mathcal{O}}

\newcommand \la {\langle}
\newcommand \ra {\rangle}

\newcommand \eps {\varepsilon}

\newcommand \Def {\stackrel{\textrm{def}}=}

\DeclareMathOperator \re {Re}
\DeclareMathOperator \im {Im}

\DeclareMathOperator \supp {supp}

\DeclareMathOperator \WF {WF}
\DeclareMathOperator \WFh {WF_{\textit{h}}}
\DeclareMathOperator \Op {Op}

\newtheorem{thm}{Theorem}

\theoremstyle{definition}

\numberwithin{lem}{section}
\numberwithin{Defn}{section}

\title
{Semiclassical resolvent estimates at trapped sets}

\author[Kiril Datchev]
{Kiril Datchev}
\address{Department of Mathematics, Massachusetts Institute of Technology, Cambridge, MA
02139-4397, U.S.A.}
\email{datchev@math.mit.edu}
\author[Andr\'as Vasy]
{Andr\'as Vasy}
\address{Department of Mathematics, Stanford University, Stanford, CA
94305-2125, U.S.A.}
\email{andras@math.stanford.edu}
\thanks{The first author is partially supported by a National Science Foundation postdoctoral fellowship, and the second author is partially supported by the National Science Foundation under
grants DMS-0801226 and DMS-1068742.}
\date{May 29, 2012}
\begin{document}

\begin{abstract}
We extend our recent results on propagation of semiclassical resolvent estimates through trapped sets when a priori polynomial resolvent bounds hold. Previously we obtained non-trapping estimates in trapping situations when the resolvent was  sandwiched between cutoffs $\chi$ microlocally supported away from the trapping: $\|\chi R_h(E+i0)\chi\| = \Oh(h^{-1})$, a microlocal version of a result of Burq and Cardoso-Vodev. We now allow one of the two cutoffs, $\tilde\chi$, to be supported at the trapped set, giving $\|\chi R_h(E+i0)\tilde\chi\| = \Oh(\sqrt{a(h)}h^{-1})$ when the a priori bound is $\|\tilde \chi R_h(E+i0)\tilde\chi\| = \Oh(a(h)h^{-1})$.
\end{abstract}

\maketitle

In this brief article we  extend the resolvent and propagation
estimates of \cite{Datchev-Vasy:Trapped}.

Let $(X,g)$ be a Riemannian manifold which is asymptotically conic or
asymptotically hyperbolic in the sense of \cite{Datchev-Vasy:Trapped},
let $V \in C_0^\infty(X)$ be real valued, let $P = h^2\Delta_g + V(x)$, where $\Delta_g\ge0$, and fix  $E>0$.

\begin{thm} \cite[Theorem 1.2]{Datchev-Vasy:Trapped}\label{t:first}  Suppose that for any $\chi_0 \in C_0^\infty(X)$ there exist $C_0, k, h_0 > 0$ such that for any $\eps > 0$, $h \in (0,h_0]$ we have
\begin{equation}\label{e:polybd}
\|\chi_0(h^2\Delta_g + V - E - i\eps)^{-1}\chi_0\|_{L^2(X) \to L^2(X)} \le C_0 h^{-k}.
\end{equation}
Let $K_E \subset T^*X$ be the set of trapped bicharacteristics at energy $E$, and suppose that $b \in C_0^\infty(T^*X)$ is identically $1$ near $K_E$. Then there exist $C_1, h_1>0$ such that for any $\eps > 0$, $h \in (0,h_1]$ we have the following nontrapping estimate:
\begin{equation}\label{e:firstconc}
\|\la r \ra^{-1/2 - \delta}(1 - \Op(b))(h^2\Delta_g + V - E - i\eps)^{-1}(1 - \Op(b))\la r \ra^{-1/2 - \delta}\|_{L^2(X) \to L^2(X)} \le C_1 h^{-1}.
\end{equation}
\end{thm}
Here by bicharacteristics at energy $E$ we mean integral curves in $p^{-1}(E)$ of the Hamiltonian vector field $H_p$ of the Hamiltonian $p = |\xi|^2 + V(x)$, and the trapped ones are those which remain in a compact set for all time. We use the notation $r = r(z) = d_g(z,z_0)$, where $d_g$ is the distance function on $X$ induced by $g$ and $z_0 \in X$ is fixed but arbitrary.

If $K_E = \varnothing$ then \eqref{e:polybd} holds with $k=1$. If $K_E \ne \varnothing$ but the trapping is sufficiently `mild', then \eqref{e:polybd} holds for some $k>1$: see \cite{Datchev-Vasy:Trapped} for details and examples. The point is that the losses in \eqref{e:polybd} due to trapping are removed when the resolvent is cutoff away from $K_E$.
Theorem~\ref{t:first}  is a more precise and microlocal version of an earlier result of  Burq \cite{Burq:Lower} and Cardoso and Vodev \cite{Cardoso-Vodev:Uniform}, but the assumption \eqref{e:polybd} is not needed in \cite{Burq:Lower, Cardoso-Vodev:Uniform}. See  \cite{Datchev-Vasy:Trapped} for additional background and references for semiclassical resolvent estimates and trapping.

In this paper we prove that an improvement over the a priori estimate \eqref{e:polybd} holds even when one of the factors of $(1 - \Op(b))$ is removed:

\begin{thm}\label{t:first2} Suppose that there exist $k >0$ and $a(h) \le h^{-k}$ such that for any $\chi_0 \in C_0^\infty(X)$ there exists $h_0 > 0$ such that for any $\eps > 0$, $h \in (0,h_0]$ we have
\begin{equation}\label{e:polybd2}
\|\chi_0(h^2\Delta_g + V - E - i\eps)^{-1}\chi_0\|_{L^2(X) \to L^2(X)} \le a(h)/h.
\end{equation}
Suppose that $b\in C_0^\infty(T^*X)$ is identically $1$ near $K_E$. Then there exist $C_1, h_1>0$ such that for any $\eps > 0$, $h \in (0,h_1]$,
\begin{equation}\label{e:firstconc2}
\|\la r \ra^{-1/2 - \delta}(1 - \Op(b))(h^2\Delta_g + V - E - i\eps)^{-1}\la r \ra^{-1/2 - \delta}\|_{L^2(X) \to L^2(X)} \le C_1 \sqrt{a(h)}/h.
\end{equation}
\end{thm}

Note that by taking adjoints, analogous estimates follow if $1 -
\Op(b)$ is placed to the other side of $(h^2\Delta_g + V - E - i\eps)^{-1}$.

Such results were proved by Burq and Zworski \cite[Theorem A]{bz} and Christianson \cite[(1.6)]{c07} when $K_E$ consists of a single hyperbolic orbit. Theorem~\ref{t:first2} implies an optimal semiclassical resolvent estimate for the example operator of \cite[\S5.3]{Datchev-Vasy:Trapped}: it improves  \cite[(5.5)]{Datchev-Vasy:Trapped} to
\[
\|\chi_0(P-\lambda)^{-1}\chi_0\| \le C \log(1/h)/h.
\]
Further, this improved estimate can be used to extend polynomial resolvent
estimates from complex absorbing potentials to analogous estimates for
damped wave equations; this is a result of Christianson, Schenk,
Wunsch and the second author \cite{csvw}.

Theorems \ref{t:first} and \ref{t:first2} follow from microlocal propagation estimates in a neighborhood of $K_E$, or more generally in a neighborhood of a suitable compact invariant subset of a bicharacteristic flow.

To state the general results, suppose $X$ is a manifold, $P \in \Psi^{m,0}(X)$ a self adjoint, order $m>0$, semiclassical pseudodifferential operator on $X$, with principal symbol $p$. For $I \subset \R$ compact and fixed, denote the characteristic set by $\Sigma=p^{-1}(I)$, and suppose that the projection to the base, $\pi\colon\Sigma \to X$, is proper (it is sufficient, for example, to have $p$ classically elliptic). Suppose that $\Gamma\Subset T^*X$ is invariant under the bicharacteristic flow in $\Sigma$. Define the {\em forward, resp.\ backward flowout} $\Gamma_+$, resp.\ $\Gamma_-$, of $\Gamma$ as the set of points $\rho \in \Sigma$, from which the backward, resp.\ forward bicharacteristic segments tend to $\Gamma$, i.e.\ for any neighborhood $O$ of $\Gamma$ there exists $T>0$ such that $-t\geq T$, resp.\ $t\geq T$, implies $\gamma(t)\in O$, where $\gamma$ is the bicharacteristic with $\gamma(0)=\rho$. Here we think of $\Gamma$ as the trapped set or as part of the trapped set, hence points in $\Gamma_-$, resp.\ $\Gamma_+$ are backward, resp.\ forward, trapped. Suppose $V$, $W$ are neighborhoods of $\Gamma$ with $\overline{V}\subset W$, $\overline{W}$ compact. Suppose also that
\begin{equation}\label{e:dynamicassumption}\begin{split}&\textrm{If $\rho\in W \setminus \Gamma_+$, resp. $\rho\in W \setminus \Gamma_-$,}\\ \textrm{then the backward, }& \textrm{resp. forward bicharacteristic from $\rho$ intersects $W\setminus \overline{V}$.}\end{split}
\end{equation}
This means that all bicharacterstics in $V$ which stay in $V$ for all time tend to $\Gamma$.

The main result of \cite{Datchev-Vasy:Trapped}, from which the other results in the paper follow, is the following:

\begin{thm}\cite[Theorem 1.3]{Datchev-Vasy:Trapped}\label{thm:through-trapped}
Suppose that $\|u\|_{H_h^{-N}} \le h^{-N}$ for some $N \in \N$ and $(P-\lambda)u=f$, $\re \lambda \in I$ and $\im \lambda \ge -\Oh(h^\infty)$. Suppose $f$ is $\Oh(1)$  on $W$, $\WFh(f)\cap \overline{V}=\emptyset$, and $u$ is  $\Oh(h^{-1})$  on $W \cap \Gamma_- \setminus \overline{V}$. Then $u$ is $\Oh(h^{-1})$  on $W\cap \Gamma_+ \setminus \Gamma$.
\end{thm}

Here we say that $u$ is $\Oh(a(h))$ at  $\rho \in T^*X$ if there exists  $B \in \Psi^{0,0}(X)$ elliptic at $\rho$ with  $\|Bu\|_{L^2} = \Oh(a(h))$. We say  $u$ is $\Oh(a(h))$ on a set $E \subset T^*X$ if it is $\Oh(a(h)))$ at each  $\rho \in E$.

Note that there is no conclusion on $u$ at $\Gamma$; typically it will be merely $\Oh(h^{-N})$ there. However, to obtain $\Oh(h^{-1})$ bounds for $u$ on $\Gamma_+$ we only needed to assume $\Oh(h^{-1})$ bounds for $u$ on $\Gamma_-$ and nowhere else.  Note also that by the propagation of singularities, if $u$ is $\Oh(h^{-1})$ at one point on any bicharacteristic, then it is such on the whole forward bicharacteristic. If $|\im \lambda| = \Oh(h^\infty)$ then the same is true for backward bicharacteristics.

In this paper we show that a (lesser) improvement on the a priori bound holds even when $f$ is not assumed to vanish microlocally near $\Gamma$:

\begin{thm}\label{thm:through-trapped2}
Suppose that $\|u\|_{H_h^{-N}} \le h^{-N}$ for some $N \in \N$ and $(P-\lambda)u=f$, $\re \lambda \in I$ and $\im \lambda \ge -\Oh(h^\infty)$. Suppose $f$ is $\Oh(1)$  on $W$, $u$ is $\Oh(a(h)h^{-1})$  on $W$, and  $u$ is  $\Oh(h^{-1})$  on $W \cap \Gamma_- \setminus \overline{V}$.
Then $u$ is $\Oh(\sqrt{a(h)}h^{-1})$  on $W\cap \Gamma_+ \setminus \Gamma$.
\end{thm}

In \cite{Datchev-Vasy:Trapped} Theorem~\ref{t:first} is deduced from Theorem~\ref{thm:through-trapped}. Theorem~\ref{t:first2} follows from Theorem~\ref{thm:through-trapped2}  by the same argument.

\begin{proof}[Proof of Theorem~\ref{thm:through-trapped2}]
The argument is a simple modification of the argument of \cite[End of
Section~4, Proof of Theorem~1.3]{Datchev-Vasy:Trapped}; we follow the
notation of this proof. Recall first from
\cite[Lemma~4.1]{Datchev-Vasy:Trapped} 
that if $U_-$ is a neighborhood of $(\Gamma_-\setminus\Gamma)
\cap(\overline{W}\setminus V)$ then
there is a neighborhood $U \subset V$ of $\Gamma$ such that if
$\alpha\in U\setminus\Gamma_+$ then the
backward bicharacteristic from $\alpha$ enters $U_-$.
Thus, if one assumes that $u$ is $\Oh(h^{-1})$ on $\Gamma_-$ and $f$ is $\Oh(1)$ on $\overline V$, it follows that that $u$ is $\Oh(h^{-1})$ on $U \setminus \Gamma_+$, provided $U_-$ is chosen small enough that $u$ is $\Oh(h^{-1})$ on $U_-$. Note also that, because $U \subset V$, $f$ is $\Oh(1)$ on $U$. We will show that $u$ is  $\Oh(\sqrt{a(h)}h^{-1})$  on $U\cap \Gamma_+ \setminus \Gamma$: the conclusion on the larger set $W\cap \Gamma_+ \setminus \Gamma$ follows by propagation of singularities.

Next, \cite[Lemma~4.3]{Datchev-Vasy:Trapped} states that if $U_1$ and $U_0$ are open sets with $\Gamma \subset U_1 \Subset U_0 \Subset U$ then there exists a nonnegative function $q \in C_0^\infty(U)$ such that 
\[q = 1 \textrm{ near } \Gamma, \qquad H_p q \le 0 \textrm{ near } \Gamma_+, \qquad H_p q< 0 \textrm{ on } \Gamma_+^{\overline{U_0}} \setminus U_1.\]
Moreover, we can take $q$ such that both $\sqrt q$ and $\sqrt{-H_pq}$ are smooth near $\Gamma_+$. 

\textbf{Remark.} The last paragraph in the proof of  \cite[Lemma~4.3]{Datchev-Vasy:Trapped} should be replaced by the following: To make $\sqrt{-H_p \tilde q}$ smooth, let $\psi(s)=0$ for $s\leq 0$, $\psi(s)=e^{-1/s}$ for $s>0$, and assume as we may that $U_\rho\cap\Sc_\rho$ is a ball with respect to a Euclidean metric (in local coordinates near $\rho$) of radius $r_\rho>0$ around $\rho$. We then choose $\varphi_{\rho}$ to behave like $\psi({r'_\rho}^2-|.|^2)$ with $r'_\rho<r_\rho$ for $|.|$ close to $r'_\rho$, bounded away from $0$ for smaller values of $|.|$, and choose $-\chi'_{\rho}$ to vanish like $\psi$ at the boundary of its support. That  sums of products of such functions have smooth square roots follows from \cite[Lemma 24.4.8]{h}.

The proof of Theorem~\ref{thm:through-trapped2} proceeds by induction: we show that if $u$ is $\Oh(h^k)$ on a sufficiently large compact subset of $U \cap\Gamma_+ \setminus \Gamma$, then $u$ is $\Oh(h^{k + 1/2})$ on $\Gamma_+^{\overline U_0} \setminus U_1$, provided $\sqrt{a(h)}h^{-1} \le C h^{k+ 1/2}$.

Now let $U_-$ be an open neighborhood of $\Gamma_+ \cap \supp q$ which is sufficiently small that $H_pq\le0$ on $U_-$ and that $\sqrt{-H_pq}$ is smooth on $U_-$. Let $U_+$ be an open neighborhood of $\supp q \setminus U_-$ whose closure is disjoint from $\Gamma_+$ and from $T^*X \setminus \overline{U}$. Define $\phi_\pm \in C^\infty(U_+ \cup U_-)$ with $\supp \phi_\pm \subset U_\pm$ and with $\phi_+^2 + \phi_-^2 = 1$ near $\supp q$.

Put
\[b \Def \phi_- \sqrt{-H_pq^2}, \qquad e \Def \phi_+^2 H_pq^2.\]
Let $Q,B,E \in \Psi^{-\infty,0}(X)$ have principal symbols $q,b,e$, and microsupports $\supp q$, $\supp b$, $\supp e$, so that
\[\frac i h [P,Q^*Q] = - B^*B + E + hF,\]
with $F \in \Psi^{-\infty,0}(X)$ such that $\WF_h'F \subset \supp dq \subset U \setminus \Gamma$. But
\begin{equation*}\begin{split}
&\frac i h \la [P,Q^*Q] u, u \ra = \frac 2 h \im \la Q^*Q(P - \lambda)u, u\ra +\frac 2 h \la Q^*Q \im \lambda u,  u \ra \\
&\ge -2h^{-1}\|Q(P - \lambda)u\|\,\|Qu\|- \Oh(h^\infty) \|u\|^2\geq -Ch^{-2}a(h) -\Oh(h^\infty),
\end{split}\end{equation*}
where we used $\im \lambda \ge -\Oh(h^\infty)$ and that on $\supp q$, $(P-\lambda)u$ is $\Oh(1)$. So
\[\|Bu\|^2 \le \la Eu,u \ra + h \la Fu, u \ra + Ch^{-2} a(h)+\Oh(h^\infty).\]
But $|\la Eu,u \ra| \le Ch^{-2}$ because $\WF'_hE \cap \Gamma_+ = \emptyset$
gives that $u$ is $\Oh(h^{-1})$ on $\WF'_hE$ by the first paragraph of the proof. Meanwhile $|\la Fu, u \ra| \le C (h^{-2} + h^{2k})$ because all points of $\WF'_hF$ are either in $U \backslash \Gamma_+$, where we know $u$ is $\Oh(h^{-1})$ from the first paragraph of the proof, or on a single compact subset of $U \cap \Gamma_+ \setminus \Gamma$, where we know that $u$ is $\Oh(h^k)$ by inductive hypothesis. Since $b= \sqrt{-H_pq^2}>0$ on $\Gamma_+^{\overline{U_0}}\setminus U_1$, we can use microlocal elliptic regularity to conclude that $u$ is $\Oh(h^{k + 1/2})$ on $\Gamma_+^{\overline U_0} \setminus U_1$, as desired.
\end{proof}

\end{document}